\documentclass{article}

\usepackage[all]{xy}
\usepackage[english]{babel}
\usepackage{amsmath}
\usepackage{amsthm}
\usepackage{amssymb}

\newcommand{\comments}[1]{} 

\newtheorem{thm}{Theorem}

\theoremstyle{definition}

\theoremstyle{remark}
\newtheorem{rem}{Remark}

\theoremstyle{remark}
\newtheorem{ex}{Example}

\def\cM{\mathcal M}

\def\cC{\mathcal C}

\def\cH{\mathcal H}
\def\cV{\mathcal V}

\def\bR{\mathbb R} 

\def\bB{\mathbb B}

\def\bN{\mathbb N}

\def\bC{\mathbb C}

\def\Ker{\operatorname{Ker}}

\def\cliff{\cC\ell}

\def\mand{\text{\ \ \ and\ \ \ }}

\def\clifford{\cC\ell}

\def\pa{\partial}

\def\bul{\;\bullet}
\def\wed{\;\wedge}
\def\pod{\underline}

\def\la{\lambda}

\begin{document}

\title{Orthogonal Appell bases for Hodge-de Rham systems in Euclidean spaces}

\author{R. L\' avi\v cka
\thanks{
Faculty of Mathematics and Physics, Charles University, Sokolovsk\'a 83, 186 75 Praha 8, Czech Republic;
Email: \texttt{lavicka@karlin.mff.cuni.cz} Tel: +420 221 913 204  Fax: +420 222 323 394 }}

\date{}

\maketitle

\begin{abstract}
Recently the Gelfand-Tsetlin construction of orthogonal bases has been explicitly described
for the spaces $\cH^s_k(\bR^m)$ of $k$-homogeneous polynomial solutions of the Hodge-de Rham system in the Euclidean space $\bR^m$ which take values in the space of $s$-vectors. In this paper, we give another construction of these bases and, mainly, we show that the bases even form complete orthogonal Appell systems. 
Moreover, we study the corresponding Taylor series expansions. 
As an application, we construct quite explicitly orthogonal bases for homogeneous solutions of an arbitrary generalized Moisil-Th\'eodoresco system in $\bR^m.$  


\end{abstract}

\section{Introduction}

In this paper, we study properties of the so-called Gelfand-Tsetlin bases for the spaces $\cH^s_k(\bR^m)$ of $k$-homogeneous polynomial solutions of the Hodge-de Rham system in the Euclidean space $\bR^m$ which take values in the space of $s$-vectors. A~construction of such bases has been recently described in \cite{GTHdR}.
In particular, we show that these bases form complete orthogonal Appell systems.
Recall that basis elements are said to possess the Appell property or form the Appell system if their derivatives are equal to a~multiple of another basis element.
We work within the context of Clifford analysis but, following \cite{BDS}, we can easily translate all the obtained results into the language of differential forms.

For a~historical account of constructions of orthogonal bases and the Appell systems in Clifford analysis, we refer to \cite{BGLS,lavCAOT}.
Let me just mention that the first construction of orthogonal bases for Clifford algebra valued monogenic functions even in any dimension is due to F. Sommen, see \cite{som,DSS}.
The Appell systems of monogenic polynomials were for the first time investigated by
H.~Malonek (see \cite{CM07}, \cite{FCM}, \cite{FM}).
In particular, in \cite{CM08}, I. Ca\c c\~ao and H. Malonek constructed an
orthogonal Appell basis  for  solutions of the
Riesz system in dimension $3$. Moreover, in \cite{lavRiesz}, it is shown that Gelfand-Tsetlin bases in this case possess the Appell property not only with respect to one variable but even with respect to all variables.
In \cite{BG}, S. Bock and K.~G\"urlebeck described orthogonal Appell bases for quaternion valued monogenic
functions in $\bR^3$. 
Later on, in \cite{lavCAOT}, complete orthogonal Appell systems have been constructed for spinor valued and Clifford algebra valued monogenic functions in any dimension.
Analogous results in Hermitian Clifford analysis have been obtained in \cite{ckH,OGinH}. 


Clifford analysis is a~systematic study of the so-called monogenic functions. 
For an account of Clifford analysis, we refer to the books \cite{DSS, GM, GHS}.
Let the vectors $e_1,\ldots,e_m$ form an orthonormal basis of the Euclidean space $\bR^m.$
Denote by $\bR_{0,m}$ the real Clifford algebra over $\bR^m$ satisfying the relations
$$e_ie_j+e_je_i=-2\delta_{ij}$$
and by $\bC_m$ the corresponding complex Clifford algebra. 
In what follows, $\clifford_m$ stands for either $\bR_{0,m}$ or $\bC_m.$
As usual, we identify a~vector $x=(x_1,\ldots,x_m)$ of $\bR^m$ with the element $x_1e_1+\cdots+x_m e_m$ of $\cliff_m$.
A~function $f$ defined and continuously differentiable in an open region $\Omega$ of $\bR^{m}$ and taking values in the Clifford algebra $\cliff_m$ is called (left) monogenic in $\Omega$ if it satisfies the Dirac equation $\pa f= 0$ in $\Omega$. Here the Dirac operator $\pa$ is defined as
\begin{equation}\label{Dirac}
\pa=e_1\pa_{x_1}+\cdots+e_m\pa_{x_m}.
\end{equation} 
Obviously, the Dirac operator $\pa$ factorizes the Laplacian $\Delta=\pa_{x_1}^2+\cdots+\pa_{x_m}^2$ in the sense that $\Delta = - \partial^2$.
Hence each monogenic function $f$ is (componentwise) harmonic, namely, it satisfies the Laplace equation $\Delta f=0$. 
As is well-known, inside the Clifford algebra $\cliff_m$ we can realize the Pin group $ Pin(m)$ as the set of finite products of unit vectors of $\bR^m$ endowed with the Cliffford multiplication.
Moreover, the group $ Pin(m)$ is a~double cover of the orthogonal group $O(m)$.
For monogenic functions $f(x),$ there are two natural actions of the group $Pin(m)$, namely, the so-called $L$-action, given by
\begin{equation}\label{Laction}
[L(r)(f)]( x)=r\,f(r^{-1} x\; r),\ r\in  Pin(m)\text{\ \ and\ \ } x\in\bR^m,
\end{equation}
and the $H$-action, given by
\begin{equation}\label{Haction}
[H(r)(f)](x)=r\;f(r^{-1} x\; r)\;r^{-1},\ r\in Pin(m)\text{\ \ and\ \ } x\in\bR^m.
\end{equation}

It turns out to be interesting to study monogenic functions taking values in a~given subspace $V$ of $\cliff_m$. 
As monogenic functions are real analytic it is  important to understand the structure of monogenic polynomials called often spherical monogenics.
Indeed, denote by $\cM_k(\bR^m,V)$ the space of 
$k$-homogeneous spherical monogenics $P:\bR^m\to V$.
Then these spaces are at least locally basic building blocks for $V$-valued monogenic functions.
To illustrate this fact, let $\bB_m$ be the unit ball in $\bR^m$ and let $L^2(\bB_m,V)$ be the space of square-integrable functions $f:\bB_m\to V$, endowed with a~scalar inner product
\begin{equation}\label{L2product}
(f,g)=\int_{\bB_m}[\bar f g]_0\;d\la^m.
\end{equation}
Here $\la^m$ is the Lebesgue measure in $\bR^m$ and, for $a\in\cliff_m,$ $\bar a$ stands for its Clifford conjugate and $[a]_0$ for its scalar part. 
Moreover, let $L^2(\bB_m,V)\cap \Ker \pa$ be 
the (real or complex) Hilbert space of $L^2$-integrable monogenic functions $f:\bB_m\to V$. 
As is well-known, 
the orthogonal direct sum
$$\bigoplus_{k=0}^{\infty} \cM_k(\bR^m,V)$$
is dense in the space $L^2(\bB_m,V)\cap \Ker \pa$. 
Hence to construct an orthogonal basis for the space $L^2(\bB_m,V)\cap \Ker \pa$ it is sufficient to find orthogonal bases in all finite dimensional subspaces $\cM_k(\bR^m,V)$.

It seems to be the most interesting when $V$ is taken as a~subspace of $\cliff_m$ invariant with respect to
the $L$-action or the $H$-action.
It is well-known (see e.g.\ \cite{DSS}) that, under the $L$-action, the spaces of homogeneous spinor valued spherical monogenics play a~role of building blocks.
On the other hand, when the given subspace  $V$ is invariant with respect to the $H$-action, we have that $V$-valued monogenic functions are just solutions of
the corresponding generalized Moisil-Th\'eodoresco system, see Section \ref{sGMT} for details. 
Generalized Moisil-Th\'eodoresco systems introduced by R. Delanghe have been recently intensively studied
(see e.g.\ \cite{BRDS, R2, DLS, DLS2, lav_isaac09, DLS3, GTHdR}).
It is well-known that the spaces $\cH^s_k(\bR^m)$ of $s$-vector valued solutions of the Hodge-de Rham system play a~role of building blocks in this case.
Moreover, for 1-vector valued functions, the Hodge-de Rham system coincides with the Riesz system, which has been carefully studied 
(see \cite{cno, R1, cac, cac10, mor09, GM10b}). 

In Section \ref{sBranch}, we recall that standard orthogonal bases for spherical harmonics are basic examples of the so-called Gelfand-Tsetlin bases. 
In Section \ref{sGT_HdR}, 
we  construct  Gelfand-Tsetlin bases for the spaces $\cH^s_k(\bR^m)$ (see Theorem \ref{GT_HdR}). 
The key step for the construction is to understand the so-called branching for Hodge-de Rham systems. 
This has been done in \cite{GTHdR} using the Cauchy-Kovalevskaya method. 
In Theorem \ref{tbranch_HdR}, we give an alternative proof based on the branching for spherical monogenics. 
The branching for spherical monogenics in $\bR^m$ is nothing else than their decomposition into spherical monogenics in $\bR^{m-1}$ multiplied by certain embedding factors and, analogously, for other cases.
Furthermore, we show that the Gelfand-Tsetlin bases form complete orthogonal Appell systems in this case and
we study the corresponding Taylor series expansions (see Theorems \ref{tAP_HdR} and \ref{tTaylor_HdR}). 
As an application, we construct quite explicitly orthogonal bases for solutions of an arbitrary generalized Moisil-Th\'eodoresco system (see Theorem \ref{OG_GMT}).

\section{Generalized Moisil-Th\'eodoresco systems}\label{sGMT}

The Clifford algebra $\cliff_m$ can be viewed naturally as the graded associative algebra
$$\cliff_m=\bigoplus_{s=0}^m\cliff_m^s.$$
Here $\cliff_m^s$ stands for the space of $s$-vectors in $\cliff_m.$ Actually, under the $H$-action, the spaces $\cliff_m^s$ are mutually inequivalent irreducible submodules of $\cliff_m$.
For a 1-vector $u$ and an $s$-vector $v,$ the Clifford product $u v$ splits into the sum of an $(s-1)$-vector $u\bullet v$ and an $(s+1)$-vector $u\wedge v.$ Indeed, we have that $uv=u\bullet v+u\wedge v$ with
$$u\bullet v=\frac 12(uv-(-1)^svu)\text{\ \ and\ \ }u\wedge v=\frac 12(uv+(-1)^svu).$$
By linearity, we extend the so-called inner product $u\bullet v$ and the outer product $u\wedge v$ for a 1-vector $u$ and an arbitrary Clifford number $v\in\cliff_m.$ 
In particular, we can split the left multiplication by a 1-vector $ x$ into the outer multiplication $( x\wed)$ and the inner multiplication $( x\bul),$ that is, $$ x=( x\wed)+( x\bul) .$$
Analogously,
the Dirac operator $\pa$ can be split also into two parts
$\pa=\pa^++\pa^-$ where
$$\pa^+P=\sum_{j=1}^m e_j\wedge(\pa_{x_j}P)\text{\ \ and\ \ }\pa^-P=\sum_{j=1}^m e_j\bullet(\pa_{x_j}P).$$
Let us remark that the operators $( x\wed)$, $( x\bul)$, $\pa^+$ and $\pa^+$ are basic examples of invariant operators with respect to the $H$-action (see \cite{DLS3}).
Furthermore, for $s$-vector valued polynomials $P$, the Dirac equation $\pa P=0$ is equivalent to the system of equations $$\pa^+ P=0,\ \ \pa^- P=0$$ we call the Hodge-de Rham system. This terminology is legitimate because, after the translation into the language of differential forms explained in \cite{BDS}, the operators $\pa^+$ and $\pa^-$ correspond to the standard de Rham differential $d$ and its codifferential $d^*$ for differential forms.

In contrast with the spinor case, it is interesting to study $V$-valued monogenic functions not only for irreducible subspaces $V$ of the Clifford algebra $\cliff_m$.
Let $V$ be an arbitrary $H$-invariant subspace of $\cliff_m$, that is, for some subset $S$ of $\{0,1,\ldots,m\}$, we have that $V=\cliff_m^S$ where
$$\cliff_m^S=\bigoplus_{s\in S}\cliff_m^s.$$
Then the corresponding generalized Moisil-Th\'eodoresco system  is defined as the Dirac equation $\pa f=0$ for $V$-valued functions $f$. 
In particular, let $\cH^s_k(\bR^m)$ be the space of $k$-homogeneous monogenic polynomials $P:\bR^m\to \cliff_m^s$.
Then the spaces $\cH^s_k(\bR^m)$ are just formed by homogeneous solutions of the Hodge-de Rham system.
Moreover, it is well-known that, under the $H$-action, the spaces $\cH^s_k(\bR^m)$ form irreducible modules 
and all non-trivial modules $\cH^s_k(\bR^m)$ are mutually inequivalent, see \cite{DLS}.
The following result  shows that the spaces $\cH^s_k(\bR^m)$ play a~role of building blocks in this case (see \cite{lav_isaac09, DLS2}).
Recall that $\cM_k(\bR^m,\cliff_m^S)$ is the space of $k$-homogeneous monogenic polynomials $P:\bR^m\to \cliff_m^S$.

\begin{thm}\label{tGMT} 
Let $S\subset\{0,1,\ldots,m\}$ and let $S'=\{s:s\pm 1\in S\}.$
Under the $H$-action of $Pin(m)$, the space $\cM_k(\bR^m,\cliff_m^S)$ decomposes into inequivalent irreducible pieces as
\begin{equation*}
\cM_k(\bR^m,\cliff_m^S)=\left(\bigoplus_{s\in S} \cH^s_k(\bR^m)\right)\oplus
\left(\bigoplus_{s\in S'}((x\wed)+\beta^{s,m}_{k-1}(x\bul))\,\cH^s_{k-1}(\bR^m)\right)
\end{equation*}
with $\beta^{s,m}_{k}=-(k+m-s)/(k+s)$. 

In particular, this decomposition is orthogonal with respect to any invariant inner product, including the $L^2$-inner product \eqref{L2product}.
\end{thm}

By Theorem \ref{tGMT}, to construct an orthogonal basis for the (real or complex) Hilbert space $L^2(\bB_m,\cliff_m^S)\cap \Ker \pa$ it obviously suffices to find orthogonal bases in the spaces $\cH^s_k(\bR^m)$, see Theorem \ref{OG_GMT} of Section \ref{sGT_HdR}.


\section{Branching for spherical harmonics and monogenics}\label{sBranch}

In this section, we describe the branching for spherical harmonics and monogenics.
But first let us recall a~standard construction of orthogonal bases for spherical harmonics. 
Denote by $ H_k(\bR^m)$ the space of complex valued harmonic polynomials in $\bR^m$ which are $k$-homogeneous. 
Let $(e_1,\ldots,e_m)$ be an orthonormal basis of the Euclidean space $\bR^m$.
Then the construction of an orthogonal basis for the space $ H_k(\bR^m)$ is based on the following decomposition (see \cite[p. 171]{GM})
\begin{equation}\label{branchHarm}
 H_k(\bR^m)=
\bigoplus_{j=0}^k F^{(k-j)}_{m,j} H_{j}(\bR^{m-1}).
\end{equation}
This decomposition is orthogonal with respect to the $L^2$-inner product on the unit ball $\bB_m$ in $\bR^m$ and
the embedding factors $F^{(k-j)}_{m,j}$ are defined as the polynomials  
\begin{equation}\label{EFHarm}
F^{(k-j)}_{m,j}(x)=\frac{(j+1)_{k-j}}{(m-2+2j)_{k-j}}\; |x|^{k-j} C^{m/2+j-1}_{k-j}(x_m/|x|),\ x\in\bR^m.
\end{equation} 
Here $x=(x_1,\ldots, x_m)$, $|x|=\sqrt{x_1^2+\cdots+x_m^2}$ and $C^{\nu}_k$ is the Gegenbauer polynomial given by
\begin{equation}\label{gegenbauer}
C^{\nu}_k(z)=\sum_{i=0}^{[k/2]}\frac{(-1)^i(\nu)_{k-i}}{i!(k-2i)!}(2z)^{k-2i}\text{\ \ with\ \ }
(\nu)_{k}=\nu (\nu+1)\cdots (\nu+k-1).
\end{equation}
The decomposition \eqref{branchHarm} shows that spherical harmonics in $\bR^m$ can be expressed in terms of spherical harmonics in $\bR^{m-1}$, that is,
for each $P\in H_k(\bR^m)$, we have that
$$P(x)=P_k(\pod x)+F^{(1)}_{m,k-1}(x) P_{k-1}(\pod x)+\cdots+F^{(k)}_{m,0}(x) P_{0}(\pod x),\ x=(\pod x, x_m)\in\bR^m$$
for some  polynomials $P_j\in H_{j}(\bR^{m-1})$. Of course, here $F^{(0)}_{m,k}=1$ and $\pod x=(x_1,\ldots, x_{m-1})$.

Applying the decomposition \eqref{branchHarm} for several times, we easily construct an orthogonal basis of the space $ H_k(\bR^m)$ by induction on the dimension $m$. Indeed, as the polynomials $(x_1\mp ix_2)^k$ form an orthogonal basis of the space $ H_k(\bR^2)$ an orthogonal basis of the space $ H_k(\bR^m)$ is formed by the polynomials
\begin{equation}\label{GTHarm}
h_{k,\mu}(x)=(x_1\mp ix_2)^{k_2}\prod^m_{r=3}F^{(k_r-k_{r-1})}_{r,k_{r-1}}
\end{equation}
where $\mu$ is an arbitrary sequence of integers $(k_{m-1}, \ldots, k_3,\pm k_2)$ such that $k=k_m\geq k_{m-1}\geq\cdots\geq k_3\geq k_2\geq 0$.

\medskip
From the point of view of representation theory, the space $ H_k(\bR^m)$ forms an irreducible module under the action of the group $Pin(m)$, defined by
$$
[h(s)(P)](x) = P(s^{-1}xs),\ s\in Pin(m)\text{\ \ and\ \ }x\in\bR^m.
$$
Moreover, the decomposition \eqref{branchHarm} is the branching of the module $H_k(\bR^m)$, that is, an irreducible decomposition of the module $H_k(\bR^m)$
under the action of the group $Pin(m-1)$. 
Here $Pin(m-1)$ is realized as the subgroup of $Pin(m)$ describing orthogonal transformations of $\bR^m$ fixing the last basis vector $e_m.$ 

In quite an analogous way, we can construct an orthogonal basis for any irreducible (finite-dimensional) $Pin(m)$-module $\cV$.  
Indeed, any such module $\cV$ always has a~multiplicity free irreducible decomposition into $Pin(m-1)$-submodules. In particular, this decomposition is orthogonal with respect to any invariant inner product given on $\cV$. The obtained orthogonal basis is called a~Gelfand-Tsetlin basis of the module $\cV$ (see \cite{GT}).
See \cite{lavCAOT} for details.

\medskip
In \cite{lavCAOT}, the Gelfand-Tsetlin construction  of orthogonal bases for  spherical monogenics is explained. 
The key ingredient of the construction is the branching for spherical monogenics, 
which is described in the following theorem (see \cite[Theorem 2.2.3, p. 315]{DSS}).  

\begin{thm}\label{tbranch_CA}
The space $\cM_k(\bR^m,\cliff_m)$ has an orthogonal decomposition
\begin{equation}
\cM_k(\bR^m,\cliff_m)=
\bigoplus_{j=0}^k X^{(k-j)}_{m,j}\cM_{j}(\bR^{m-1},\cliff_m).
\end{equation}
Here the embedding factors $X^{(k-j)}_{m,j}$ are defined as the polynomials
\begin{equation}\label{EF_CA}
X^{(k-j)}_{m,j}(x)=F^{(k-j)}_{m,j}+
\frac{j+1}{m+2j-1}\;F^{(k-j-1)}_{m,j+1}\;\pod x e_m,\ x=(\pod x, x_m)\in\bR^m
\end{equation}
with $\pod x=x_1e_1+\cdots+x_{m-1} e_{m-1}$, $F^{(k-j)}_{m,j}$ defined in \eqref{EFHarm} and $F^{(-1)}_{m,k+1}=0$.
\end{thm}

In \cite{lavCAOT}, it is shown that, for spherical harmonics and monogenics in any dimension, Gelfand-Tsetlin bases even form complete orthogonal Appell systems. Moreover, this is not a~coincidence but the consequence of the construction of the Gelfand-Tsetlin bases and the fact 
that $\pa_{x_m}$ is obviously an invariant operator under the action of the subgroup $Pin(m-1)$. 
In the next section, we prove analogous results for Hodge-de Rham systems.


\section{Gelfand-Tsetlin bases for Hodge-de Rham systems}\label{sGT_HdR}

In this section, we  construct Gelfand-Tsetlin bases for the spaces $\cH^s_k(\bR^m)$ of $k$-homogeneous monogenic polynomials $P:\bR^m\to\cliff_m^s$. 
Here $\cliff_m^s$ stands for the space of $s$-vectors in $\cliff_m.$

\begin{rem}
Obviously, we have that $\cH^s_k(\bR^m)=\{0\}$ for $s\in\{0,m\}$ and $k\geq 1.$ 
In the case when $\cliff_m=\bR_{0,m}$ (resp.\ $\bC_m$), we have that
$\cH^0_0(\bR^m)=\bR$ (resp.\ $\bC$) and $\cH^m_0(\bR^m)=\bR\, e^*_M$ (resp.\ $\bC\, e^*_M$) with $e^*_M=e_me_{m-1}\cdots e_1.$ 
\end{rem}

As we know,  the space $\cH^s_k(\bR^m)$ forms an irreducible module 
under the $H$-action of the Pin group $Pin(m)$.
Moreover, all non-trivial modules $\cH^s_k(\bR^m)$ are mutually inequivalent, see \cite{DLS}.
The key step for constructing the Gelfand-Tsetlin bases is to understand the branching of the module $\cH^s_k(\bR^m)$. 
 
\begin{thm}\label{tbranch_HdR} 
Let $m\geq 3$, $s=0,\ldots,m$ and $k\in\bN_0$.
Denote by $N^{s,m}_k$ the set of pairs $(t,j)\in \{0,\ldots, m-1\}\times\{0,\ldots,k\}$ such that 
$t\in\{s-1,s\}$ and, if $t\in\{0,m-1\}$ then $j=0$.
Then, under the $H$-action of $Pin(m-1)$, we have the following multiplicity free irreducible decomposition
\begin{equation}\label{branch_HdR}
\cH^s_k(\bR^m)=\bigoplus_{(t,j)\in N^{s,m}_k}\;X^{s,t,m}_{k,j}\;\cH^t_j(\bR^{m-1}).
\end{equation}
Here the embedding factors $X^{s,t,m}_{k,j}$ are defined as the polynomials
$$
X^{s,t,m}_{k,j}(x)=X^{(k-j)}_{m,j}(x)e_m^{s-t}+\alpha\;X^{(k-1-j)}_{m,j+1}(x)(\beta^{t-s}(\pod x\;\wedge)+\beta^{t-s+1}\;(\pod x\;\bullet))\;e_m^{s-t+1}
$$ 
where $x=(\pod x,x_m)\in\bR^m$, $X^{(k-j)}_{m,j}$ are given in \eqref{EF_CA}, $X^{(-1)}_{m,k+1}=0$,\\ $\beta=-(j+m-1-t)/(j+t)$ and
$$
\alpha=
\left\{
\begin{array}{ll}
-(j+1)/(m+2j-1)& \text{unless\ \ \ }t=0,m-1;\medskip\\
0& \text{if\ \ \ }t=0,m-1.
\end{array}
\right.
$$
In particular, this decomposition is orthogonal with respect to any invariant inner product, including the $L^2$-inner product \eqref{L2product}.

\end{thm}

\begin{proof} In \cite{GTHdR}, it is shown by the Cauchy-Kovalevskaya method. 
Now we give an alternative proof. It is a~well-known fact from representation theory that, under the $H$-action of $Pin(m-1)$,
the module $\cH^s_k(\bR^m)$ possesses the decomposition 
\begin{equation*}
\cH^s_k(\bR^m)=\bigoplus_{(t,j)\in N^{s,m}_k}\tilde\cH^t_j
\end{equation*}
into irreducible submodules $\tilde\cH^t_j$ equivalent to $\cH^t_j(\bR^{m-1})$. 
To get \eqref{branch_HdR} we need to describe explicitly the pieces $\tilde\cH^t_j$ in the decomposition.
To do this, we recall that, 
by Theorem \ref{tbranch_CA}, we have that
\begin{equation}\label{branch_CA2}
\cM_k(\bR^m,\cliff_m)=
\bigoplus_{j=0}^k X^{(k-j)}_{m,j}\cM_{j}(\bR^{m-1},\cliff_m).
\end{equation}
As $\cliff_m=\cliff_{m-1}\oplus e_m\;\cliff_{m-1}$ each space $\cM_{j}(\bR^{m-1},\cliff_m)$ in \eqref{branch_CA2} decomposes further as
$\cM_{j}(\bR^{m-1},\cliff_m)=\cM_{j}(\bR^{m-1},\cliff_{m-1})\oplus e_m\cM_{j}(\bR^{m-1},\cliff_{m-1})$. 
Moreover, by Theorem \ref{tGMT}, any space $\cM_{j}(\bR^{m-1},\cliff_{m-1})$ decomposes into $Pin(m-1)$-irreducible submodules as
$\cM_j(\bR^{m-1},\cliff_{m-1})=$ 
\begin{equation*}
\left(\bigoplus_{t=0}^{m-1} \cH^t_j(\bR^{m-1})\right)\oplus
\left(\bigoplus_{t=1}^{m-2}((\pod x\wed)+\beta^{t,m-1}_{j-1}(\pod x\bul))\,\cH^t_{j-1}(\bR^{m-1})\right).
\end{equation*}
As a~result, we have decomposed the space $\cM_k(\bR^m,\cliff_m)$ into $Pin(m-1)$-irreducible pieces
\begin{equation}\label{pieces}
X^{(k-j)}_{m,j} e_m^r\; \cH^t_j(\bR^{m-1})\text{\ and\ }X^{(k-j)}_{m,j}((\pod x\wed)+\beta^{t,m-1}_{j-1}(\pod x\bul)) e_m^r\; \cH^t_{j-1}(\bR^{m-1})
\end{equation}
where $r=0,1$ and, resp., $t=0,\ldots,m-1$, $j=0,\ldots,k$ and $t=1,\ldots,m-2$, $j=1,\ldots,k$. 
Now it is easy to find the submodule $\tilde\cH^t_j$ equivalent to $\cH^t_j(\bR^{m-1})$ inside $\cH^s_k(\bR^m)$.
Indeed, by \eqref{pieces}, it is sufficient to choose a~constant $\tilde\alpha$ such that, for a~$Pin(m-1)$-invariant
$$
X^{s,t,m}_{k,j}(x)=X^{(k-j)}_{m,j}(x)e_m^{s-t}+\tilde\alpha\;X^{(k-1-j)}_{m,j+1}(x)((\pod x\;\wedge)+\beta^{t,m-1}_{j}\;(\pod x\;\bullet))\;e_m^{s-t+1},
$$ 
we have that $X^{s,t,m}_{k,j}\;\cH^t_j(\bR^{m-1})\subset\cH^s_k(\bR^m)$. 
As we know that $$X^{s,t,m}_{k,j}\;\cH^t_j(\bR^{m-1})\subset\cM_k(\bR^m,\cliff_m)$$ it is sufficient to take the constant $\tilde\alpha$ such that
the piece $X^{s,t,m}_{k,j}\;\cH^t_j(\bR^{m-1})$ contains only $s$-vector valued polynomials. But, recalling the definition \eqref{EF_CA} of the factors $X^{(k-j)}_{m,j}$,  this is not difficult to do. 
\end{proof}

Using Theorem \ref{tbranch_HdR}, we easily construct orthogonal (or even Gelfand-Tsetlin) bases for the spaces $\cH^s_k(\bR^m)$ by induction on the dimension $m$.

\begin{ex}\label{GT_HdR_2} 
First we construct orthogonal bases for Hodge-de Rham systems in dimension 2. 
Indeed, the following statements are obvious.

\medskip\noindent
(i) For $s\in\{0,2\}$, an orthogonal basis for $\cH^s_0(\bR^2)$ is formed by the unique element $f^s_0=e^{s,0}$ with $e^{0,0}=1$ and  $e^{2,2}=e_{21}$.

\medskip\noindent
(ii) Let $\cliff_2=\bR_{0,2}$. Then,
for $k\in\bN_0$, an orthogonal basis for $\cH^1_k(\bR^2)$ consists of two polynomials $f^{\pm 1}_k(x)=(x_1-e_{12}x_2)^k e^{1,\pm 1}$ with 
$e^{1,1}=e_1$ and $e^{1,-1}=e_2$.

\medskip\noindent
(iii) Let $\cliff_2=\bC_2$. Then,
for $k\in\bN_0$, an orthogonal basis for $\cH^1_k(\bR^2)$ consists of two polynomials $f^{\pm 1}_k(x)=(x_1\mp ix_2)^k e^{1,\pm 1}$ with 
$e^{1,\pm 1}=e_1\mp i e_2$.

\medskip\noindent
(iv) Otherwise, the spaces $\cH^s_k(\bR^2)$ are trivial.
\end{ex}

\begin{thm}\label{GT_HdR} 
Let $m\geq 3$. 
Denote by $I^{s,m}_k$ the set of pairs $(\nu,\mu)$ such that the sequences 
$\nu=(s_{m-1},\ldots,s_3,t_2)$ and $\mu=(k_{m-1},k_{m-1},\ldots,k_2)$
of integers satisfy
$(t_2,k_2)\in\{(0,0),(2,0)\}\cup\{(\pm1,k)|\ k\in\bN_0\}$ and, 
for each $r=3,\ldots,m$, $(s_{r-1},k_{r-1})\in N^{s_r,r}_{k_r}$  
where $k_m=k$, $s_m=s$ and $s_2=|t_2|$.

\medskip\noindent
Then an orthogonal (or even Gelfand-Tsetlin) basis of $\cH^s_k(\bR^m)$ is formed by the polynomials
$$
f^{s,\nu}_{k,\mu}=X^{s,s_{m-1},m}_{k,k_{m-1}} X^{s_{m-1},s_{m-2},m-1}_{k_{m-1},k_{m-2}}\cdots X^{s_3,s_2,3}_{k_3,k_2}\;f^{t_2}_{k_2},
\ (\nu,\mu)\in I^{s,m}_k.
$$
Here the embedding factors $X^{s,t,m}_{k,j}$ are given in Theorem \ref{tbranch_HdR} and $f^{t_2}_{k_2}$ in Example~\ref{GT_HdR_2}.
\end{thm}

The basis elements $f^{s,\nu}_{k,\mu}$ have the following Appell property.

\begin{thm}\label{tAP_HdR}
Let $m\geq 3$ and let $f^{s,\nu}_{k,\mu}$ be the basis elements of the spaces $\cH^s_k(\bR^m)$ given in Theorem \ref{GT_HdR} 
with $\nu=(s_{m-1},\ldots,s_3,t_2)$, $\mu=(k_{m-1},k_{m-1},\ldots,k_2)$ and $s_2=|t_2|$.
Then we have that

\begin{itemize}
\item[(i)] $\pa_{x_m}  f^{s,\nu}_{k,\mu}=0$ for $k=k_{m-1}$;

\item[(ii)] $\pa_{x_m}  f^{s,\nu}_{k,\mu}=k\;  f^{s,\nu}_{k-1,\mu}$ for $k>k_{m-1}$;

\item[(iii)] $\pa_{t_2}^{k_2}\;\pa^{k_3-k_2}_{x_3}\cdots\pa^{k-k_{m-1}}_{x_m}\;  f^{s,\nu}_{k,\mu}=k!\;e^{s,\nu}$\\
with 
$e^{s,\nu}=e_m^{s-s_{m-1}}\cdots e_3^{s_3-s_2} e^{s_2, t_2}$ and
\begin{equation}\label{derHdR2}
\pa_{t_2}=
\left\{
\begin{array}{ll}
(1/2)(\pa_{x_1}+e_{12}\pa_{x_2})& \text{if\ \ \ }\cliff_m=\bR_{0,m}\text{\ and\ }t_2=\pm 1;\medskip\\
(1/2)(\pa_{x_1}\pm i\pa_{x_2})& \text{if\ \ \ }\cliff_m=\bC_m\text{\ and\ }t_2=\pm 1.
\end{array}
\right.
\end{equation}
Note that $k_2=0$ unless $t_2=\pm 1$.
\end{itemize}

\end{thm}

\begin{proof}
Using standard formulas for Gegenbauer polynomials, it is easy to verify that, for $k>j$, $\pa_{x_m} F^{(k-j)}_{m,j}=k\;F^{(k-1-j)}_{m,j}$
and $F^{(0)}_{m,k}=1$ and, consequently, we get that, for $k>j$, $\pa_{x_m} X^{(k-j)}_{m,j}=k\;X^{(k-j-1)}_{m,j}$ and $X^{(0)}_{m,k}=1$ (see \cite{lavCAOT}).
Hence we conclude that, for $k>j$,  $\pa_{x_m} X^{s,t,m}_{k,j}=k\;X^{s,t,m}_{k-1,j}$ and $X^{s,t,m}_{k,k}=e_m^{s-t}$, which completes easily the proof.
\end{proof}

\begin{rem}\label{resnu} 
Let $s\in\{0,\ldots,m\}$ and put $s_m=s$. Denote by $J^{s,m}$ the set of sequences 
$\nu=(s_{m-1},\ldots,s_3,t_2)$ of integers such that, for each $r=2,\ldots,m-1$, 
$$0\leq s_r\leq r\mand s_{r+1}-1\leq s_{r}\leq s_{r+1}$$ 
where $s_2=|t_2|$.
Obviously, the set $\{e^{s,\nu}|\ \nu\in J^{s,m}\}$ is a~basis of the space $\cliff_m^s$ of $s$-vectors.
Here $e^{s,\nu}$ are given in Theorem \ref{tAP_HdR}.
Then each $a\in\cliff_m^s$ can be uniquely  written as
$$a=\sum_{\nu\in J^{s,m}} a^{\nu} e^{s,\nu}$$ for some (real or complex) numbers $a^{\nu}$.
\end{rem}

To summarize, we have constructed a~complete orthogonal Appell system for the Hilbert space
$L^2(\bB_m,\cliff_m^s)\cap \Ker \pa$ of $L^2$-integrable monogenic functions $g:\bB_m\to\cliff_m^s$.
Indeed, using Theorems \ref{GT_HdR} and \ref{tAP_HdR}, we easily obtain the following result.

\begin{thm}\label{tTaylor_HdR} Let $m\geq 3$, let $\bB_m$ be the unit ball in $\bR^m$ and let $s=0,\ldots,m$. 

\begin{itemize}
\item[(a)] Then an orthogonal basis of the space $L^2(\bB_m,\cliff_m^s)\cap \Ker \pa$ is formed by the polynomials
$f^{s,\nu}_{k,\mu}$ for $k\in\bN_0$ and $(\nu,\mu)\in I^{s,m}_k$.
Here the basis elements $f^{s,\nu}_{k,\mu}$ are defined in Theorem \ref{GT_HdR}.

\item[(b)]
Each function $g\in L^2(\bB_m,\cliff_m^s)\cap \Ker \pa$ has a~unique orthogonal series expansion
\begin{equation}\label{taylor_HdR}
g = \sum_{k=0}^{\infty}\; \sum_{(\nu,\mu)\in I^{s,m}_k} \mathbf{t}^{s,\nu}_{k,\mu}(g)\; f^{s,\nu}_{k,\mu}
\end{equation}
for some complex coefficients $\mathbf{t}^{s,\nu}_{k,\mu}(g)$. 
In addition, by Remark \ref{resnu}, we have that $$g=\sum_{\nu\in J^{s,m}} g^{\nu} e^{s,\nu}$$ for some complex functions $g^{\nu}$ on $\bB_m$.
Then, for $(\nu,\mu)\in I^{s,m}_k$, it holds that
\begin{equation}\label{coeff_HdR}
\mathbf{t}^{s,\nu}_{k,\mu}(g)
=\frac{1}{k!}\;\pa_{t_2}^{k_2}\;\pa^{k_3-k_2}_{x_3}\cdots\pa^{k-k_{m-1}}_{x_m}\; g^{\nu}(x)|_{x=0}
\end{equation}
where $\pa_{t_2}$ is defined in \eqref{derHdR2}.
\end{itemize}
\end{thm}

For a~function $g\in L^2(\bB_m,\cliff_m^s)\cap \Ker \pa$, we call the orthogonal series expansion \eqref{taylor_HdR} its generalized Taylor series.

Now we construct orthogonal bases for solutions of an arbitrary generalized Moisil-Th\'eodoresco  system. 
It is easy to see that, using Theorems \ref{tGMT} and \ref{GT_HdR}, we obtain the following result.

\begin{thm}\label{OG_GMT}
Let $S$ be a~subset of $\{0,1,\ldots,m\}$ and let $S'=\{s:s\pm 1\in S\}$. 
Then an orthogonal basis of the Hilbert space $L^2(\bB_m,\cliff_m^S)\cap \Ker \pa$ is formed by the polynomials
$$f^{s,\nu}_{k,\mu}\text{\ \ for\ }s\in S,\ k\in\bN_0\text{\ and\ }(\nu,\mu)\in I^{s,m}_k,$$
together with the polynomials
$$((x\wed)+\beta^{s,m}_{k-1}(x\bul))\,f^{s,\nu}_{k-1,\mu}\text{\ \ for\ }s\in S',\ k\in\bN\text{\ and\ }(\nu,\mu)\in I^{s,m}_{k-1}.$$
Here $\cliff_m^S=\bigoplus_{s\in S}\cliff_m^s$ and $\beta^{s,m}_{k}=-(k+m-s)/(k+s)$.
\end{thm}

\subsection*{Acknowledgments}

I am grateful to V. Sou\v cek for useful conversations.
The financial support from the grant GA 201/08/0397 is gratefully acknowledged.




\begin{thebibliography}{99}
\def\topsep{0pt}
\def\parsep{0pt plus 5pt minus 1pt}
\def\itemsep{-0.5ex} 
\small

\bibitem{BRDS} R. Abreu Blaya, J. Bory Reyes, R. Delanghe and F. Sommen, Generalized Moisil-Th\'eodoresco systems and Cauchy integral decompositions, Int. J. Math. Math. Sci.,
Vol. 2008, Article ID746946, 19 pages.

\bibitem{BGLS} S. Bock, K. G\"urlebeck, R. L\' avi\v cka and V. Sou\v cek, The Gelfand-Tsetlin bases for spherical monogenics in dimension 3, 
arXiv:1010.1615v2 [math.CV], 2010, to appear in Rev. Mat. Iberoamericana 

\bibitem{BG} S. Bock and K. G\"urlebeck, On a~generalized Appell
system and monogenic power series, Math. Methods Appl. Sci. \textbf{33} (2010),
394--411.

%
\bibitem{BDS} F. Brackx, R. Delanghe and F. Sommen, {Differential forms and/or multi-vector
functions}, CUBO \textbf{7} (2005), 139-170.
%

\bibitem{ckH} F.\ Brackx, H.\ De Schepper, R.\ L\'{a}vi\v{c}ka, V.\ Sou\v{c}ek,
{The Cauchy-Kovalevskaya	Extension Theorem in Hermitean Clifford Analysis}, J. Math. Anal. Appl. \textbf{381} (2011), 649-660. 

\bibitem{OGinH} F. Brackx, H. De Schepper, R.~L\'avi\v cka and  V. Sou\v cek, Gelfand-Tsetlin Bases of Orthogonal Polynomials in Hermitean Clifford Analysis, 
arXiv:1102.4211v1 [math.CV], 2011, to appear in Math. Methods Appl. Sci.

\bibitem{cac10} I. Ca\c c\~ao, {Complete orthonormal sets of polynomial solutions of the Riesz and Moisil-Teodorescu systems in $\bR^3$}, Numer. Algor. {\bf 55} (2010), 191-203.


%
 \bibitem{CM08} I. Ca\c c\~ao and H.~R. Malonek, On a complete set of hypercomplex Appell polynomials,
Proc. ICNAAM 2008, (T. E. Timos, G. Psihoyios, Ch. Tsitouras, Eds.), AIP Conference Proceedings
1048, 647-650.
%
 \bibitem{cac} I. Ca\c c\~ao, {Constructive approximation by monogenic polynomials}, PhD thesis, Univ. Aveiro, 2004.
%
\bibitem{cno} J. Cnops, {Reproducing kernels of spaces of vector valued monogenics}, Adv. appl. Clifford alg. \textbf{6} 1996 (2), 219-232.
%
\bibitem{R1} R. Delanghe, {On homogeneous polynomial solutions of the Riesz system and their harmonic potentials},
Complex Var. Elliptic Equ.  \textbf{52}  (2007),  no. 10-11, 1047--1061.
%
\bibitem{R2} R. Delanghe, {On homogeneous polynomial solutions of generalized Moisil-Th\'eodoresco systems
in Euclidean space}, CUBO \textbf{12} (2010), 145-167.
%
%
\bibitem{DSS} R. Delanghe, F. Sommen, V. Sou\v cek, {Clifford Algebra and Spinor-valued Functions}, Mathematics and Its Applications 53, Kluwer Academic Publishers, 1992. 
%
\bibitem{DLS} R. Delanghe, R. L\'avi\v cka and V. Sou\v cek, {On polynomial solutions of generalized Moisil-Th\'eodoresco systems and Hodge-de Rham systems}, 
Adv. appl. Clifford alg. {\bf 21} (2011), 521-530.  
%
\bibitem{DLS2} R. Delanghe, R. L\'avi\v cka and V. Sou\v cek, {The Fischer decomposition for Hodge-de Rham systems in Euclidean spaces}, 
arXiv:1012.4994v1 [math.CV], 2010, to appear in Math. Methods Appl. Sci. 
%
\bibitem{DLS3} R. Delanghe, R. L\'avi\v cka and V. Sou\v cek, {The Howe duality for Hodge systems}, In: Proceedings of 18th International Conference on the Application of Computer Science and Mathematics in Architecture and Civil Engineering  (ed. K. Gürlebeck and C. Könke),
Bauhaus-Universität Weimar, Weimar, 2009.

\bibitem{GTHdR} R. Delanghe, R.~L\'avi\v cka and  V. Sou\v cek, The Gelfand-Tsetlin bases for Hodge-de Rham systems in Euclidean spaces, 
arXiv:1012.4998v1 [math.CV], 2010, to appear in Math. Methods Appl. Sci. 

%
\bibitem{FCM} M.~I. Falc\~ao, J.~F. Cruz and H.~R. Malonek, Remarks on the generation of monogenic
functions, Proc. of the 17-th International Conference on the Application of Computer Science and
Mathematics in Architecture and Civil Engineering, ISSN 1611-4086 (K. G\"urlebeck and C. K\"onke,
eds.), Bauhaus-University Weimar, 2006.
%
\bibitem{FM} M.~I. Falc\~ao and H.~R. Malonek, Generalized exponentials through Appell sets in $\bR^{n+1}$ and
Bessel functions, Numerical Analysis and Applied Mathematics (T.E. Simos, G. Psihoyios, and Ch.
Tsitouras, eds.), AIP Conference Proceedings, vol. 936, American Institute of Physics, 2007, pp.
750-753 (ISBN: 978-0-7354-0447-2).

\bibitem{GT} I.~M.~Gelfand and M.~L.~Tsetlin, { Finite-dimensional representations of groups of orthogonal
matrices}, Dokl. Akad. Nauk SSSR \textbf{71} (1950), 1017--1020 (Russian). English transl.\ in: I. M.
Gelfand, Collected papers, Vol.\ II, Springer-Verlag, Berlin, 1988, 657--661.
%
%
\bibitem{GM} J. E. Gilbert and M. A. M. Murray, {Clifford Algebras and Dirac Operators in Harmonic Analysis},
Cambridge University Press, Cambridge, 1991.
%
\bibitem{GHS} K.\ G\"{u}rlebeck, K.\ Habetha, W.\ Spr\"{o}\ss ig, 
{Holomorphic functions in the plane and $n$-dimensional space. Translated from the 2006 German original, with cd-rom (Windows and UNIX)}, Birkh\"{a}user Verlag, Basel, 2008.

\bibitem{GM10b}
K. G\"urlebeck and J. Morais, {Real-Part Estimates for Solutions of the Riesz System in $\bR^3$}, 
to appear in Complex Var. Elliptic Equ.

\bibitem{lav_isaac09} R. L\'avi\v cka, {The Fischer Decomposition for the $H$-action and Its Applications}, 
In: Hypercomplex analysis and applications, I. Sabadini and F. Sommen (eds.), Trends in Mathematics, Springer Basel AG, 2011, pp. 139-148.

\bibitem{lavRiesz} R.~L\'avi\v cka, Generalized Appell property for the Riesz system in dimension 3, 
In: ICNAAM 2011, Halkidiki, Greece, 2011 (eds. T.E. Simos, G. Psihoyios, Ch. Tsitouras), AIP Conf. Proc. {\bf 1389} (2011), pp. 291-294. 
 
\bibitem{lavCAOT} R.~L\'avi\v cka, Complete orthogonal Appell systems for spherical monogenics, arXiv:1106.2970v2 [math.CV], 2011, to appear in Complex Anal. Oper. Theory. 


\bibitem{CM07} H.~R. Malonek, M.~I. Falc\~ao, Special monogenic polynomials\,-\,properties and applications. In Numerical Analysis and Applied Mathematics (T.E.
Simos, G. Psihoyios, and Ch. Tsitouras, eds.), AIP Conference Proceedings, vol. 936. American Institute of
Physics: Melville, NY, 2007; 764-767.
 
\bibitem{mor09} J. Morais, Approximation by homogeneous polynomial solutions of the Riesz system in $\bR^3$, PhD thesis, Bauhaus-Univ., Weimar, 2009.

\bibitem{som} F. Sommen, Spingroups and spherical means III, Rend. Circ. Mat. Palermo
(2) Suppl. No 1 (1989), 295-323.

\end{thebibliography}
\end{document}